\documentclass[12pt]{article}
\usepackage{amsmath,amssymb,amsthm, comment}

\newtheorem{theorem}{Theorem}[section]
\newtheorem{thmy}{Theorem}

\newtheorem{lemma}[theorem]{Lemma}
\newtheorem{corollary}[theorem]{Corollary}

\newcommand\numberthis{\addtocounter{equation}{1}\tag{\theequation}}

\def\barr{\begin{array}}
\def\earr{\end{array}}

\title{A characterization of $A_5$ by its average order}
\author{Marius T\u arn\u auceanu}
\date{December 15, 2022}

\begin{document}

\maketitle

\begin{abstract}
Let $o(G)$ be the average order of a finite group $G$. M. Herzog, P. Longobardi and M. Maj \cite{7} showed that if $G$ is non-solvable and $o(G)=o(A_5)$,
then $G\cong A_5$. In this note, we prove that the equality $o(G)=o(A_5)$ does not hold for any finite solvable group $G$. Consequently, up to isomorphism, 
$A_5$ is determined by its average order.
\end{abstract}

{\small
\noindent
{\bf MSC2000\,:} Primary 20D60; Secondary 20D10, 20F16.

\noindent
{\bf Key words\,:} average order, sum of element orders, solvable group.}

\section{Introduction}
Given a finite group $G$, we denote by $\psi(G)$ the sum of element orders of $G$ and by $o(G)$ the average order of $G$, that is
\begin{equation}
\psi(G)=\sum\limits_{x\in G}o(x) \mbox{ and } o(G)=\frac{\psi(G)}{|G|}\,.\nonumber
\end{equation}

In the last years there has been a growing interest in studying the pro\-per\-ties of these functions and their relations with the structure of $G$ (see for example \cite{1}-\cite{7}, \cite{9}-\cite{10}, \cite{12} and \cite{14}-\cite{16}).

In \cite{9}, A. Jaikin-Zapirain uses the average order to determine a lower bound for the number of conjugacy classes of a finite $p$-group/nilpotent group. He also suggests the following question: "\textit{Let $G$ be a finite {\rm(}$p$-{\rm)}group and $N$ be a normal {\rm(}abelian{\rm)} subgroup of $G$. Is it true that $o(G)\geq o(N)^{\frac{1}{2}}$}\,\,?\,". Recently, E.I. Khukhro, A. Moret\' o and M. Zarrin proved the following result (see Theorem 1.2 of \cite{10}):

\begin{thmy}
Let $c>0$ be a real number and $p\geq \frac{3}{c}$ be a prime. Then there exists a finite $p$-group $G$ with a normal abelian subgroup $N$ such that $o(G)<o(N)^c$.
\end{thmy}

Note that Theorem A provides a negative answer to Jaikin-Zapirain's question even if we replace the exponent $\frac{1}{2}$ with any positive real number $c$. In the same paper \cite{10}, the authors posed the following conjecture:

\bigskip\noindent{\bf Conjecture.} {\it Let $G$ be a finite group and suppose that $o(G)<\frac{211}{60}=o(A_5)$. Then $G$ is solvable.}
\bigskip

It has been confirmed by M. Herzog, P. Longobardi and M. Maj \cite{7}.

\begin{thmy}
Let $G$ be a finite group. If $o(G)<\frac{211}{60}$\,, then $G$ is solvable. Moreover, if $G$ is a non-solvable finite group, we have $o(G)=\frac{211}{60}$ if and only if $G\cong A_5$.
\end{thmy}

This constitutes the starting point for the current work. Our main result is as follows.

\begin{theorem}
If $G$ is a finite solvable group, then $o(G)\neq\frac{211}{60}$\,.
\end{theorem}

It is now clear that Theorems B and 1.1 lead to the following characterization of $A_5$.

\begin{corollary}
Up to isomorphism, $A_5$ is the unique finite group $G$ such that $o(G)=\frac{211}{60}$\,.
\end{corollary}

For the proof of our results, we need the next theorem.

\begin{thmy}
Let $G$ be a finite group and $n_d(G)$ be the number of elements of order $d$ in $G$, $\forall\, d\in\mathbb{N}$. Then the following statements hold:
\begin{itemize}
\item[{\rm 1)}]{\rm (T.J. Laffey \cite{11})} If $p$ is a prime divisor of $|G|$ and $G$ is not a $p$-group, then
\begin{equation}
n_p(G)\leq\frac{p}{p+1}\,|G|-1.\nonumber
\end{equation}
\item[{\rm 2)}]{\rm (G.A. Miller \cite{13})} If $G$ is non-abelian, then
\begin{equation}
n_2(G)\leq\frac{3}{4}\,|G|-1.\nonumber
\end{equation}
\end{itemize}
\end{thmy}

Most of our notation is standard and will usually not be repeated here. Elementary notions and results on groups can be found in \cite{8}.

\section{Proofs of the main results}

We start with the following easy but important lemma.

\begin{lemma}
Let $G$ be a finite group of order $n$ and $n_d(G)$ be the number of elements of order $d$ in $G$, $\forall\, d\in\mathbb{N}^*$. If $o(G)=\frac{211}{60}$\,, then
\begin{equation}
4n_2(G)+3n_3(G)+2n_4(G)+n_5(G)\geq\frac{149}{60}\,n-5.
\end{equation}
\end{lemma}

\begin{proof}
Since $n=\sum_{d\in\mathbb{N}^*}n_d(G)$, we deduce that
\begin{align*}
\psi(G)&=1+2n_2(G)+\ldots+6n_6(G)+\ldots\\
&\geq 1+2n_2(G)+\ldots+5n_5(G)+6(n_6(G)+n_7(G)+\ldots)\\
&=1+2n_2(G)+\ldots+5n_5(G)+6(n-1-n_2(G)-\ldots-n_5(G))\\
&=6n-5-(4n_2(G)+3n_3(G)+2n_4(G)+n_5(G)).
\end{align*}Since $o(G)=\frac{211}{60}$\,, it follows that $\psi(G)=\frac{211}{60}\,n$. Therefore we have
\begin{equation}
\frac{211}{60}\,n\geq 6n-5-(4n_2(G)+3n_3(G)+2n_4(G)+n_5(G)),\nonumber
\end{equation}which is equivalent to (1), completing the proof.
\end{proof}

Our second lemma shows that the average order of a finite nilpotent group which is not a $p$-group cannot be certain rational numbers.

\begin{lemma}
Let $G$ be a finite nilpotent group. If $G$ is not a $p$-group, then $o(G)\not\in\left\{\frac{61}{15}\,,\frac{211}{60}\right\}$.
\end{lemma}

\begin{proof}
We write $G$ as a direct product of its Sylow subgroups
\begin{equation}
G=G_1\times\cdots\times G_k,\nonumber
\end{equation}where $k\geq 2$ and $G_i\in{\rm Syl}_{p_i}(G)$, for all $i=1,...,k$. Since the function $o$ is multiplicative, we get
\begin{equation}
o(G)=o(G_1)\cdots o(G_k).\nonumber
\end{equation}Let $|G_i|=p_i^{n_i}$, $i=1,...,k$, and assume that $p_1<...<p_k$. Then for each $i$ we have
\begin{equation}
o(G_i)\geq\frac{1+p_i(p_i^{n_i}-1)}{p_i^{n_i}}=p_i-\frac{p_i-1}{p_i^{n_i}}\geq p_i-\frac{p_i-1}{p_i}>1.\nonumber
\end{equation}For $k\geq 3$, we get
\begin{equation}
o(G)\geq o(G_1)o(G_2)o(G_3)\geq\!\left(2{-}\frac{1}{2}\right)\!\!\left(3{-}\frac{2}{3}\right)\!\!\left(5{-}\frac{4}{5}\right)\!=\frac{147}{10}>\frac{61}{15}>\frac{211}{60}\,.\nonumber
\end{equation}If $k=2$ and $(p_1,p_2)\neq(2,3)$, then
\begin{equation}
o(G)=o(G_1)o(G_2)\geq\min\!\left\{\!\!\left(3{-}\frac{2}{3}\right)\!\!\left(5{-}\frac{4}{5}\right)\!\!,\!\left(2{-}\frac{1}{2}\right)\!\!\left(5{-}\frac{4}{5}\right)\!\!\right\}\!=\frac{63}{10}>\frac{61}{15}\,.\nonumber
\end{equation}Thus, we can suppose that $(p_1,p_2)=(2,3)$. If $(n_1,n_2)\neq(1,1)$, then
\begin{equation}
o(G)=o(G_1)o(G_2)\geq\min\!\left\{\!\!\left(2{-}\frac{1}{4}\right)\!\!\left(3{-}\frac{2}{3}\right)\!\!,\!\left(2{-}\frac{1}{2}\right)\!\!\left(3{-}\frac{2}{9}\right)\!\!\right\}\!=\frac{49}{12}>\frac{61}{15}\,.\nonumber
\end{equation}Finally, for $(n_1,n_2)=(1,1)$ we have $G\cong C_6$ and so
\begin{equation}
o(G)=\frac{7}{2}\not\in\left\{\frac{61}{15}\,,\frac{211}{60}\right\}\!,\nonumber
\end{equation}as desired.
\end{proof}

We are now able to prove our main result. Note that a search with GAP shows that there is no solvable group $G$ of order less than $1000$ such that $o(G)=\frac{211}{60}$\,.

\bigskip\noindent{\bf Proof of Theorem 1.1.}
Assume that there exists a finite solvable group $G$ of order $n$ with $o(G)=\frac{211}{60}$\,. Then $60\psi(G)=211n$, implying that $n=2^23^{\alpha_2}5^{\alpha_3}\cdots$ with $\alpha_2,\alpha_3\geq 1$. Let $H$ be a Sylow $2$-subgroup of $G$. We distinguish the following two cases.

\medskip
\hspace{5mm}\noindent{\bf Case 1.} $H\cong C_4$\vspace{2mm}

Then $G$ has a normal $2$-complement $K$ of order $m=3^{\alpha_2}5^{\alpha_3}\cdots$. First of all, we prove that
\begin{equation}
n_4(G)=2m=\frac{1}{2}\,n\,.
\end{equation}Suppose that $n_4(G)<2m$. Since $\frac{n_4(G)}{2}$ divides $m$, it follows that $\frac{n_4(G)}{2}\leq\frac{m}{3}$ and so $n_4(G)\leq\frac{2}{3}\,m=\frac{1}{6}\,n$. This implies that $n_2(G)\leq\frac{1}{12}\,n$\,. Also, since all elements of odd order of $G$ are contained in $K$, we get $n_3(G),n_5(G)< m=\frac{1}{4}\,n$\,. Thus we have
\begin{equation}
4n_2(G)+3n_3(G)+2n_4(G)+n_5(G)<\left(\frac{1}{3}+\frac{3}{4}+\frac{1}{3}+\frac{1}{4}\right)n=\frac{5}{3}\,n,\nonumber
\end{equation}which together with (1) lead to
\begin{equation}
\frac{149}{60}\,n-5<\frac{5}{3}\,n,\nonumber
\end{equation}that is $n\leq 6$, a contradiction.

Next we prove that
\begin{equation}
n_2(G)=m=\frac{1}{4}\,n\,.
\end{equation}Suppose that $n_2(G)<m$. Since every subgroup of order $2$ of $G$ is contained in at least three Sylow $2$-subgroups, we infer that $n_2(G)\leq\frac{m}{3}=\frac{1}{12}\,n$\,. Similarly, we get
\begin{equation}
4n_2(G)+3n_3(G)+2n_4(G)+n_5(G)<\left(\frac{1}{3}+\frac{3}{4}+1+\frac{1}{4}\right)n=\frac{7}{3}\,n\nonumber
\end{equation}and so
\begin{equation}
\frac{149}{60}\,n-5<\frac{7}{3}\,n\nonumber
\end{equation}by (1). Thus $n\leq 33$, a contradiction.

The equalities (2) and (3) show that any two Sylow $2$-subgroups of $G$ intersect trivially and any element of $G$ of even order is a $2$-element. We deduce that $G$ is a Frobenius group and consequently its kernel $K$ is nilpotent. On the other hand, we have
\begin{equation}
\frac{211}{60}\,n=\psi(G)=\psi(K)+2n_2(G)+4n_4(G)=\psi(K)+\frac{5}{2}\,n,\nonumber
\end{equation}implying that
\begin{equation}
\psi(K)=\frac{61}{60}\,n=\frac{61}{15}\,m,\nonumber
\end{equation}or equivalently
\begin{equation}
o(K)=\frac{61}{15}\,.\nonumber
\end{equation}This contradicts Lemma 2.2.

\medskip
\hspace{5mm}\noindent{\bf Case 2.} $H\cong C_2^2$\vspace{2mm}

Since $G$ is solvable, it has a Hall $2$-subgroup $K$ of order $m=3^{\alpha_2}5^{\alpha_3}\cdots$. We have the next two subcases.

\medskip
\hspace{10mm}\noindent{\bf Subcase 2.1.} $K$ is normal in $G$\vspace{2mm}

Then $K$ contains all elements of odd order of $G$. By using Theorem C, 1), for $p=3$, it follows that $n_3(G)\leq\frac{3}{4}\,m=\frac{3}{16}\,n$. Also, we have $n_3(G)+n_5(G)\leq m-1=\frac{n}{4}-1$. Then
\begin{equation}
3n_3(G)+n_5(G)=2n_3(G)+(n_3(G)+n_5(G))\leq\frac{3}{8}\,n+\frac{1}{4}\,n-1=\frac{5}{8}\,n-1,\nonumber
\end{equation}implying that
\begin{equation}
4n_2(G)+3n_3(G)+2n_4(G)+n_5(G)\leq 4n_2(G)+\frac{5}{8}\,n-1.\nonumber
\end{equation}From (1) we get
\begin{equation}
n_2(G)\geq\frac{223}{480}\,n-1.
\end{equation}

Let $H_i$, $i=1,2,3$, be the subgroups of order $2$ of $H$. Then $G=KH_1\cup KH_2\cup KH_3$ and so $n_2(G)=3n_2(KH_i)$, for any $i=1,2,3$. On the other hand, we have $n_2(KH_i)\mid m$ by Sylow's theorems. If $n_2(KH_i)<m$, then $n_2(KH_i)\leq\frac{m}{3}$\,, which lead to
\begin{equation}
n_2(G)\leq m=\frac{1}{4}\,n.
\end{equation}Now, from (4) and (5), we obtain
\begin{equation}
\frac{223}{480}\,n-1\leq\frac{1}{4}\,n,\nonumber
\end{equation}that is $n\leq 4$, a contradiction. Thus $n_2(KH_i)=m$, for all $i=1,2,3$, and $n_2(G)=3m=\frac{3}{4}\,n$. Then, by Theorem C, 2), it follows that $G$ is abelian. Since $o(G)=\frac{211}{60}$\,, this contradicts Lemma 2.2.

\medskip
\hspace{10mm}\noindent{\bf Subcase 2.2.} $K$ is not normal in $G$\vspace{2mm}

Then $K$ has two or four conjugates in $G$. Let $C={\rm Core}_G(K)$. Since $[G:C]$ divides $4!=24$, we get $[G:C]=12$ and so $[K:C]=3$. We observe that $A_4$ is the unique group of order $12$ containing more than one subgroup of order $3$. Thus $G/C\cong A_4$ and $K$ has four conjugates in $G$, say $K_i$, $i=1,..,4$. Then $G$ possesses
\begin{equation}
\left|\bigcup_{i=1}^4K_i\right|=m+3\left(m-\frac{m}{3}\right)=3m\nonumber
\end{equation}elements of odd order. Also, we have
\begin{align*}
n_3(G)&=n_3\!\!\left(\bigcup_{i=1}^4K_i\!\right)\leq n_3(K_1)+3\left|\bigcup_{i=2}^4(K_i\setminus K_1)\right|\\
&\leq\frac{3m}{4}+3\left(m-\frac{m}{3}\right)=\frac{11}{4}\,m=\frac{11}{16}\,n\numberthis\label{eqn}
\end{align*}by applying Theorem C, 1), for the group $K_1$ and $p=3$.

Assume that $n_2(G)=m$. Then the number $s_2$ of Sylow $2$-subgroups of $G$ is at least $\frac{m}{3}$\,. Since all Sylow $2$-subgroups of $G$ are contained in $CH$ and $|CH|=4\cdot\frac{m}{3}$, we get $s_2\mid\frac{m}{3}$ and therefore $s_2=\frac{m}{3}$\,. This implies again that $G$ is a Frobenius group, a contradiction. Thus
\begin{equation}
n_2(G)\leq\frac{m}{3}=\frac{n}{12}\,.\nonumber
\end{equation}Since $C$ includes all Sylow $5$-subgroups of $G$, it results that
\begin{equation}
n_5(G)\leq|C|-1=\frac{m}{3}-1=\frac{n}{12}-1\,.\nonumber
\end{equation}Then
\begin{equation}
4n_2(G)+3n_3(G)+2n_4(G)+n_5(G)\leq\frac{n}{3}+3n_3(G)+\frac{n}{12}-1=3n_3(G)+\frac{5}{12}\,n-1\nonumber
\end{equation}which together with (1) lead to
\begin{equation}
3n_3(G)\geq\left(\frac{149}{60}-\frac{5}{12}\right)n-4,\nonumber
\end{equation}that is
\begin{equation}
n_3(G)\geq\frac{31}{45}\,n-\frac{4}{3}\,.
\end{equation}Now, from (6) and (7), we obtain
\begin{equation}
\frac{31}{45}\,n-\frac{4}{3}\leq\frac{11}{16}\,n,\nonumber
\end{equation}that is $n\leq 960$, a contradiction.

This completes the proof.$\qed$

\vspace*{3ex}\small

\hfill
\begin{minipage}[t]{5cm}
Marius T\u arn\u auceanu \\
Faculty of  Mathematics \\
``Al.I. Cuza'' University \\
Ia\c si, Romania \\
e-mail: {\tt tarnauc@uaic.ro}
\end{minipage}

\end{document}